\documentclass[12pt]{amsart}
\usepackage{graphicx}
\usepackage{amsmath,amsthm, amsfonts,amssymb, stmaryrd,yfonts,pxfonts,pifont,eufrak, bbm}
\newtheorem{Theorem}{Theorem}[section]

\newtheorem{Cor}{Corollary}
 \newtheorem{Lemma}{Lemma}
 
 \newtheorem{ex}{Example}
 \newtheorem{Proposition}{Proposition}
 \theoremstyle{definition}
 
 \theoremstyle{remark}
 \newtheorem{Remark}[Lemma]{Remark}
 \numberwithin{equation}{subsection}

\newcommand{\ra}{\rightarrow}

\begin{document}
\title[INDUCED DYNAMICS ON THE HYPERSPACES]{INDUCED DYNAMICS ON THE HYPERSPACES}%
\author{Puneet Sharma}
\address{Department of Mathematics, I.I.T. Jodhpur, Old Residency Road, Ratnada, Jodhpur-342011, INDIA}%
\email{puneet.iitd@yahoo.com}%




\begin{abstract}
In this paper, we study the dynamics induced by finite commutative
relation. We prove that the dynamics generated by such a non-trivial
collection cannot be transitive/super-transitive and hence cannot
exhibit higher degrees of mixing. As a consequence we establish that
the dynamics induced by such a collection on the hyperspace endowed
with any admissible hit and miss topology cannot be transitive and
hence cannot exhibit any form of mixing. We also prove that if the
system is generated by such a commutative collection, under suitable
conditions the induced system cannot have dense set of periodic
points. In the end we give example to show that the induced dynamics
in this case may or may not be sensitive.
\end{abstract}
\maketitle

\section{INTRODUCTION}

\subsection{Motivation:}
Dynamical were introduced to investigate different physical and
natural phenomenon occurring in nature. Using the theory of
dynamical systems, mathematical models for various physical/natural
phenomenon were developed and long term behavior of the natural
phenomenon/systems were investigated. While logistic maps were used
to develop the population model for any species, Lorentz system of
differential equations was used for developing mathematical models
for weather predictions. Since then, dynamical systems (both
discrete and continuous) have found applications in various branches
of science and engineering and various phenomenon occurring in a
variety of disciplines have been investigated. In some of the recent
studies, it has been observed that many systems observed in
different branches of science and engineering can be investigated
using set-valued dynamics ( c.f. \cite{rr,dd,srcg,sg}). While
\cite{srcg} used set-valued dynamics to study handwheel force
feedback for lanekeeping assistance, \cite{sg} used set-valued
dynamics to investigate the collective dynamics of an electron and a
nuclei. These examples suggest that the dynamics of different
systems evolving in various disciplines of science and engineering
can be modeled using set-valued dynamics. Thus, it is important to
study the set-valued dynamics induced by a continuous self map which
inturn can help characterizing the dynamics of a general dynamical
system. Many researchers have addressed the problem and many of the
questions in this direction have been answered
\cite{ba,do,ap,pa1,pa2,hrf}. In the process, the dynamical behavior
of a system and its corresponding set-valued counterpart has been
investigated and several interesting results have been obtained. In
\cite{ba,hrf}, authors proved that while weakly mixing and
topological mixing on the two spaces are equivalent, transitivity on
the base space need not imply transitivity on the hyperspace.
Interesting results relating the topological entropy of the two
spaces have been obtained \cite{do}. In \cite{pa1}, Sharma and Nagar
investigated some of the natural questions arising from this
setting. They investigated the influence of each of the individual
units and the role of the underlying hyperspace topology in
determining the dynamics induced by a continuous self map on a
general topological space. In the process they investigated
properties like dense periodicity, transitivity, weakly mixing and
topological mixing. They also investigated notions like sensitive
dependence on initial conditions, topological entropy, Li-Yorke
chaos, existence of Li-Yorke pairs, existence of horseshoe and the
corresponding results were established. Investigating the inverse of
the problem stated, they also investigated the behavior of an
individual component of the system, given the dynamical behavior of
the induced system on the
hyperspace \cite{pa1,pa2}.\\

Generalizing the stated problem, Nagar and Sharma\cite{ap}
investigated the dynamics induced by a finite collection of
continuous self maps. They derived necessary and sufficient
conditions for the induced collective dynamics to exhibit various
dynamical notions. They introduced the notion of super-transitivity,
super-weakly mixing and super-topological mixing for investigating
the dynamics induced on the hyperspace. They proved that
super-transitivity of a relation is necessary to induce transitivity
on the hyperspace. They proved that for any finite relation $F$ on
the space $X$, the induced map on the hyperspace $\mathcal{K}(X)$ is
weakly mixing (topologically mixing) if and only if the relation $F$
is super-weakly mixing (super-topologically mixing). However, the
existence of such systems was left open and some natural questions
were raised. When does a system induced by a non-trivial family
(family of more than one map) exhibit transitivity? When does the
dynamics induced by such a family exhibit stronger forms of mixing?
Can the dynamics induced by such a system exhibit dense set of
periodic points? In this paper, we try to answer some of the
questions raised in \cite{ap}. We now give some of the preliminaries
needed to
establish our results.\\

\subsection{Dynamics of a Relation}

Let $(X,d)$ be a compact metric space and let
$F=\{f_1,f_2,\ldots,f_k\}$ be a finite collection of continuous self
maps on $X$. The pair $(X,F)$ generates a multi-valued dynamical
system via the rule $F(x)=\{f_1(x),f_2(x),\ldots,f_k(x)\}$. For
convenience, we denote such systems by $(X,F)$. Such a system
generalizes the concept of the dynamical system generated by a
single map $f$. The objective of a study of dynamical system is to
study the orbit $\{F^n(x): n \in \mathbb{N} \}$ of an arbitrary
point $x$, where $F^n(x)= \{f_{i_1}(x)\circ
f_{i_2}(x)\circ\ldots\circ f_{i_n}(x):  1\leq i_1,
i_2,\ldots,i_n\leq k \} $ is the n-fold composition of $F$. We now
define some of the basic dynamical notions for such a system.\\

A point $x$ is called periodic for if there exists $n\in \mathbb{N}$
such that $x\in F^n(x)$. The least such $n$ is known as the period
of the point $x$. The relation $F$ is transitive if for any pair of
non-empty open sets $U,V$ in $X$, there exists $n\in \mathbb{N}$
such that $F^n(U)\cap V \neq \phi$. The relation $F$ is
super-transitive if for any pair of non-empty open sets $U,V$ in
$X$, there exists $x \in U, n\in \mathbb{N}$ such that
$F^n(x)\subset V$. The relation $F$ is said to be \textit{weakly
mixing} if for every two pair of non-empty open sets $U_1, U_2$ and
$V_1, V_2$, there exists a natural number $n$ such that $F^n(U_i)
\bigcap V_i \neq \phi$, $i=1,2$. The relation $F$ is said to be
\textit{super-weakly mixing} if for every two pair of non-empty open
sets $U_1, U_2$ and $V_1, V_2$, there exists $x_i \in U_i$ and a
natural number $n$ such that $F^n(x_i) \subseteq V_i$, $i=1,2$. The
relation $F$ is said to be \textit{topologically mixing} if for
every pair of non-empty open sets $U, V$ there exists a natural
number $K$ such that $F^n(U) \bigcap V \neq \phi$ for all $n \geq
K$. The relation $F$ is said to be \textit{super-topologically
mixing} if for every pair of non-empty open sets $U, V$ there exists
$K \in \mathbb{N}$ such that for each natural number $n \geq K$,
there exists $x_n \in U$ such that $F^n(x_n) \subseteq V$. A
relation $F$ is sensitive if there exists a $\delta>0$ such that for
each $x\in X$ and each $\epsilon>0$, there exists $n\in \mathbb{N}$
and $y\in X$ such that $d(x,y)<\epsilon$ but
$d_H(F^n(x),F^n(y))>\delta$. It may be noted that as $F^n(x)$ and
$F^n(y)$ are subsets (and need not be elements) of $X$, metric $d$
cannot be used to measure the distance between them. As
$F^n(x),F^n(y)$ are elements in the hyperspace and the metric $d_H$
is a natural extension of the metric $d$, $d_H$ is used to compute
the distance between any $F^n(x)$ and $F^n(y)$ (refer to sec. 1.3
for the details). Incase the relation $F$ is map, the above
definitions coincide with the known
dynamical notions of a system. See \cite{bc,bs,de,ap} for details.\\


\subsection{Some Hyperspace Topologies} Let $(X,\tau)$ be a Hausdorff topological space and let $\Psi$ be a
subfamily of all non-empty closed subsets of $X$. Let $\Psi$ be
endowed with topology $\Delta$, where the topology $\Delta$ is
generated using the topology $\tau$ of $X$. Then the pair
$(\Psi,\Delta)$ is called the hyperspace associated with $(X,\tau)$.
A hyperspace topology is called admissible if the map $x\rightarrow
\{x\}$ is continuous. In this paper, we are interested in only
admissible hyperspaces $(\Psi,\Delta)$ associated with $(X,\tau)$.
We now give some of the notations and terminologies used in the
article.

$CL(X) = \{ A\subset X : A \text{~~is non-empty and closed}\}$\\
$\mathcal{K}(X)=\{A\in CL(X) : A \text{~~is compact}\}$\\
$\mathcal{F}(X)=\{A\in CL(X) : A \text{~~is finite} \}$\\
$\mathcal{F}_n(X)=\{A\in CL(X) : |A|=n, \text{~~where~~} |A| \text{~~denotes number of elements in A} \}$\\
$E^{-}= \{A\in \Psi : A \cap E\neq \phi\}$\\
$E^{+}= \{A\in \Psi : A \subset E\}$\\
$E^{++}= \{A\in \Psi : \exists ~~\epsilon > 0 \text{~~such that~~}
S_{\epsilon}(A) \cap E\}$,\\ where $S_{\epsilon}(A)=\bigcup
\limits_{a\in A} S(a,\epsilon)$, where $S_{\epsilon}(x)=\{y\in X:
d(x,y)<\epsilon\}$\\

We now give some of the standard hyperspace topologies.\\

{\bf Vietoris Topology:} For any $n\in\mathbb{N}$ and any finite
collection of non-empty open sets $\{U_1,U_2,\ldots,U_n\}$, define

$<U_1,U_2,\ldots,U_n> = \{A\in \Psi : A\subset \bigcup
\limits_{i=1}^n U_i, ~~~ A\bigcap U_i \neq \phi ~~~ \forall i\}$

Varying $n\in \mathbb{N}$ and $U_i$ over the collection of all
non-empty open sets of $X$ generates a basis for a topology on the
hyperspace
known as the Vietoris topology. \\

{\bf Hausdorff Metric Topology:} Let $(X,d)$ be a metric space. For
any $A,B \in \Psi$, define $d_H(A,B)= \inf \{ \epsilon>0 :
A\subseteq S_{\epsilon}(B) \text{~~~and~~~} B\subseteq
S_{\epsilon}(A)\}$. Then $d_H$ defines a metric on $\Psi$ and the
topology generated is known as the Hausdorff metric topology. It may
be noted that $d_H(\{x\},\{y\})=d(x,y)$ and hence the metric $d_H$
preserves the metric $d$ on $X$. It is known that the Hausdorff
metric topology and the Vietoris topology coincide, incase $X$ is a
compact metric space. See \cite{be,nai} for details.\\

{\bf Hit and Miss Topology:} Let $\Phi \subseteq CL(X)$ be a
subfamily of all non-empty closed subsets of $X$. The \textit{Hit
and Miss topology} generated by the family $\Phi$ is the topology
generated by sets of the form $U^-$ where $U$ is open in $X$, and
$(E^c)^{+}$ with $E \in \Phi$, where $E^c$ denotes the complement of
$E$. As a terminology, $U$ is called the \textit{hit set} and any
member $E$ of $\Phi$ is referred as the \textit{miss set}.\\

{\bf Hit and Far-Miss Topology:} Let $(X,d)$ be a metric space and
let $\Phi$ be a given collection of closed subsets of $X$. The
\textit{Hit and Far Miss topology} generated by the collection
$\Phi$ is the topology generated by the sets of the form $U^-$ where
$U$ is open in $X$ and $(E^c)^{++}$ with $E \in \Phi$.

Here a sub-basic open set in the hyperspace hits an open set $U
\subset X$ or far misses the complement of a member of $\Phi$ and
hence forms a Hit and Far Miss topology.  It is known that any
topology on the hyperspace is of Hit and Miss or Hit and Far-Miss
type \cite{nai}.\\

{\bf Lower and Upper Vietoris Topology:} Consider the collection of
sets of the form $U^-$ in the hyperspace, where $U$ is non-empty
open set in $X$. The smallest topology on the hyperspace in which
all the sets of the form $U^-$ considered are open is known as the
Lower Vietoris topology.\\

Consider the collection of sets of the form $U^+$ in the hyperspace,
where $U$ is non-empty open set in $X$. The smallest topology on the
hyperspace in which all the sets of the form $U^+$ considered are
open is known as the Upper Vietoris topology. It can be seen that
the Vietoris topology equals the join of Upper Vietoris and Lower
Vietoris topology, and is an example of a Hit and Miss topology.\\

A detailed survey on the hyperspace topologies may be found in
\cite{be,mi,mio,nai}.

\subsection{Dynamics Induced by a Relation}

Let $(X,F)$ be a dynamical system generated by a finite family of
continuous self maps on $X$, say $\{f_1,f_2,\ldots,f_k\}$. For any
$\Psi\subset CL(X)$, the collection $\Psi$ is said to be admissible
with respect to $F$ if $F(A) ~(=\bigcup \limits_{i=1}^k f_i(A))\in
\Psi$ for all $A\in \Psi$. It may be noted that any collection
$\Psi$ admissible to $F$ generates a map $\overline{F}$ on $\Psi$
via the rule  $\overline{F}(A)= F(A)$. Consequently, endowing $\Psi$
with any suitable hyperspace topology (such that the map
$\overline{F}$ is continuous) generates a dynamical system on the
hyperspace.\\

It is interesting to investigate the relation between the dynamical
behavior of $(X,F)$ and the induced system $(\Psi,\overline{F})$. A
special case when the family $F$ is a singleton has been
investigated by several authors and a lot of work in this direction
has already been done\cite{ba,do,ap,pa1,pa2,hrf}. It was proved that
while weakly mixing and topological mixing on the two spaces are
equivalent, transitivity on the base space need not imply
transitivity on the hyperspace\cite{ba,hrf}. While \cite{do}
investigated the topological entropy of induced system of all
non-empty compact subsets of $X$, \cite{pa1} investigated the
problem for a general hyperspace endowed with a general hyperspace
topology. They discussed various dynamical notions like dense
periodicity, transitivity, weakly mixing, topological mixing and
topological entropy. They also investigated metric related dynamical
notions like equicontinuity, sensitivity, strong sensitivity and
Li-Yorke chaoticity\cite{pa1,pa2}. Authors extended their studies to
the general case when the dynamics on $X$ is generated by a finite
family and derived the necessary and sufficient conditions for the
induced system to exhibit various dynamical notions\cite{ap}. In the
process, they discussed properties like dense periodicity,
transitivity, weakly mixing and topological mixing. They introduced
the notion of super-weakly mixing and super-topological mixing for
relations to study the induced maps on the hyperspace. Authors
proved that for any finite relation $F$ on the space $X$, the
induced map on the hyperspace $\mathcal{K}(X)$ is weakly mixing
(resp. topologically mixing) if and only if the relation $F$ is
super-weakly mixing (resp. super-topologically mixing). However, the
existence of any such collection of maps was not confirmed and the
problem was left open. For the sake of
completion, we mention some of their results below.\\

%
%



\begin{Theorem}\cite{ap}
Let $\beta$ be any base for the topology on $X$ and $\Delta$ be the
topology on $\Psi \subseteq CL(X)$ such that $U^+$ is non-empty and
$U^+ \in \Delta$ for every $U \in \beta$. Then, the system
$(\Psi,\overline{F})$ is transitive implies that the relation
$(X,F)$ is super-transitive.
\end{Theorem}

\begin{Theorem}\cite{ap}
Let $\mathcal{F}(X) \subseteq \Psi.$ If $F$ is super-weakly mixing,
then $\overline{F}$ is weakly mixing. The converse holds if there
exists a base $\beta$ for topology on $X$ such that $U^+ \in \Delta$
for every $U \in \beta$.
\end{Theorem}

\begin{Theorem}\cite{ap}
Let $\mathcal{F}(X) \subseteq \Psi.$ If $F$ is super-topologically
mixing, then $\overline{F}$ is topologically mixing. The converse
holds if there exists a base $\beta$ for topology on $X$ such that
$U^+ \in \Delta$ for every $U \in \beta$.
\end{Theorem}

In this paper, we answer some of the questions raised in \cite{ap}
when $X$ is a compact metric space. We prove that the dynamics
generated by any such commutative family cannot be
transitive/super-transitive and hence cannot generate any of the
complex mixing notions on the hyperspace. We also prove that
dynamics induced by such a collection on the hyperspace cannot
exhibit dense set of periodic points. In the end we give example to
prove that the dynamics generated by such family may be sensitive.

\section{Main Results}

\begin{Lemma}
Let $(X,F)$ be a dynamical system generated by a finite commutative
family of continuous self maps on X. Then, $F$ is super-transitive
if and only if for each non-empty open set $U$ and any point $x\in
X$, there exists $u\in U$ and a sequence $(n_i)$ in $\mathbb{N}$
such that $F^{n_i}(u)\ra \{x\}$ in $(\mathcal{K}(X),d_H)$.
\end{Lemma}

\begin{proof}

Let $F=\{f_1,f_2,\ldots,f_k\}$ and let $d$ be a compatible metric on
$X$. Let $U$ be a non-empty open subset of $X$ and let $V_1=
S_1(x)$. As $F$ is super-transitive, there exists $x_1 \in U$ and
$n_1 \in \mathbb{N}$ such that $F^{n_1}(x_1)\subset V_1$. As
$F^{n_1}(x)=\{f_{i_1}\circ f_{i_2}\circ\ldots\circ f_{i_{n_1}}:
1\leq i_1,i_2, \ldots,i_{n_1}\leq k\}$ and each $f_{i_1}\circ
f_{i_2}\circ\ldots\circ f_{i_{n_1}}$ is continuous (and are finitely
many maps), there exists a neighborhood $U_1$ of $x_1$ such that
$U_1\subset U$ and $F^{n_1}(\overline{U_1})\subset V_1$.

Let $V_2=S_{\frac{1}{2}}(x)$. As $F$ is super-transitive (applying
transitivity to $U_1$ and $V_2$), there exists $x_2 \in U_1$ and
$n_2 \in \mathbb{N}$ such that $F^{n_2}(x_2)\subset V_2$. As
$F^{n_2}(x)=\{f_{i_1}\circ f_{i_2}\circ\ldots\circ f_{i_{n_2}}:
1\leq i_1,i_2, \ldots,i_{n_2}\leq k\}$ and each $f_{i_1}\circ
f_{i_2}\circ\ldots\circ f_{i_{n_2}}$ is continuous (and number of
maps are finite), there exists a neighborhood $U_2$ of $x_2$ such
that $U_2\subset U_1$ and $F^{n_2}(\overline{U_2})\subset V_2$.

Inductively, let $U_r, V_r$ of non-empty open sets in $X$ such that
$U_r\subset U_{r-1}$,  $V_r = S_{\frac{1}{r}}(x)$ and
$F^{n_r}(\overline{U_r})\subset V_r$. Let $V_{r+1} =
S_{\frac{1}{r+1}}(x)$. As $F$ is super-transitive (applying
super-transitivity to the pair $(U_r, V_{r+1})$), there exists
$x_{r+1} \in U_r$ and $n_{r+1} \in \mathbb{N}$ such that
$F^{n_{r+1}}(x_{r+1})\subset V_{r+1}$. One again, as
$F^{n_{r+1}}(x)=\{f_{i_1}\circ f_{i_2}\circ\ldots\circ
f_{i_{n_{r+1}}}: 1\leq i_1,i_2, \ldots,i_{n_{r+1}}\leq k\}$ and each
$f_{i_1}\circ f_{i_2}\circ\ldots\circ f_{i_{n_{r+1}}}$ is continuous
(and number of maps are finite), there exists a neighborhood
$U_{r+1}$ of $x_{r+1}$ such that $U_{r+1}\subset U_r$ and
$F^{n_{r+1}}(\overline{U_{r+1}})\subset V_{r+1}$.

Consequently, we obtain a nested sequence of open sets $(U_r)$
contained in $U$ such that $d_H(F^{n_r}(u),x)<\frac{1}{r}$ for any
$u\in U_r$.  As $\overline{U_i}$ is a decreasing sequence of
non-empty compact subsets of $X$, $A= \cap \overline{U_i}\subset U$
is non-empty. Let $u \in A$, then $u \in U_r$ and hence $F^{n_r}(u)
\subset V_r$ for all $r$. Consequently,
$d_H(F^{n_r}(u),x)<\frac{1}{r}$ for all $r$ and hence $F^{n_i}(u)\ra
\{x\}$.

Conversely, let $U,V$ be a pair of non-empty open sets ion $X$. For
any $\{v\}\in <V>$, there exists $u\in U$ and a sequence $(n_i)$ in
$\mathbb{N}$ such that $F^{n_i}(u)\ra \{v\}$. Consequently there
exists $r\in \mathbb{N}$ such that $F^{n_k}(u)\subset V ~ \forall
~k\geq r$ and hence $F$ is super-transitive.
\end{proof}

\begin{Remark}
It may be noted that transitivity of a system generated by single
map $f$ is equivalent to the existence of a dense orbit.
Consequently, the above result is trivially true when the family $F$
is a singleton. Thus the result above is a generalization of the
known result to the case when the system is generated using more
than one map. It may be noted that $F^n$ is a union of repeated
application of the maps $\{f_1,f_2,\ldots,f_k\}$ (n times) in all
possible orders and continuity of each component map of $F^n$
guarantees extension of a behavior at a point to a similar behavior
in the neighborhood of the point. Further, as the number of maps
constituting $F^n(x)$ is finite at each iteration, a neighborhood
exhibiting similar behavior for all components of $F^n$ is ensured
and hence the existence of $U_i$ is guaranteed. It is worth
mentioning that the commutativity of the family $F$ is not used for
establishing the result. Hence the result is true even when the
generating family $F$ is non-commutative.
\end{Remark}

\begin{Proposition}
Let $(X,F)$ be a dynamical system generated by a finite commutative
family of continuous self maps on X. If $F$ is super-transitive,
then $F$ is a singleton.
\end{Proposition}

\begin{proof}
Let $U$ be an non-empty open set in $X$ and $\epsilon>0$ be a real
number. Let $V$ be a non-empty open set and let $v\in V$. By Lemma,
there exists $u\in U$ and a sequence $(n_i)$ such that
$F^{n_i}(u)\ra v$. As each $f_i$ is continuous, there exists
$\delta>0$ such that $d(x,y)<\delta$ implies
$d(f_i(x),f_i(y))<\frac{\epsilon}{2}$ for all $i=1,2,\ldots,n$. As
$F^{n_i}(u)\ra v$, for any two distinct elements $f,g$ of $F$,
$f^{n_i}(u)\ra v$ and $g^{n_i}(u)\ra v$. Further, as
$d(F^{n_i}(u))\ra 0$ (where $d(A)$ denotes the diameter of the set
$A$), there exists $r\geq 1$ such that the relation $d(f^{n_i}(u),
g^{n_i}(u))<\delta$ and $d(f\circ g^{n_i-1}(u), g^{n_i}(u))<\delta$
is true for all $i\geq r$.\\

Consequently, $d(f^{n_i+1}(u),f \circ
g^{n_i}(u))<\frac{\epsilon}{2}$ and $d(f\circ g^{n_i}(u),
g^{n_i+1}(u))<\frac{\epsilon}{2}$ for all $i\geq r$. Using triangle
inequality we get $d(f^{n_i+1}(u), g^{n_i+1}(u))<\epsilon$ for all
$i\geq r$. Consequently, $(f^{n_i+1}(u)), (g^{n_i+1}(u))$ are
parallel sequences and hence have the same limit(say $y$). Also
$f^{n_i}(u)\ra v$ and $g^{n_i}(u)\ra v$ implies $f^{n_i+1}(u)\ra
f(v)$ and $g^{n_i+1}(u)\ra g(v)$. Consequently, $f(v)=g(v)$. As the
argument holds for any open set $V$, the points at which $f$ and $g$
coincide is dense in $X$. Hence $f=g$.
\end{proof}

\begin{Remark}
The above proof establishes that the system induced by more than one
map cannot be super-transitive. The proof establishes that under
stated conditions, if $(f^{n_i}(x))$ and $(g^{n_i}(x))$ are parallel
then $(f^{n_i+1}(x))$ and $(g^{n_i+1}(x))$ are also parallel and
hence have the same limit. Further as $f^{n_i}(x),g^{n_i}(x)$
converge to $v$, $f(v)$ and $g(v)$ are unique limit points of
$(f^{n_i+1}(x))$ and $(g^{n_i+1}(x))$ respectively. Consequently $f$
and $g$ coincide on a dense set and hence are equal.
\end{Remark}

\begin{Remark}
The above result proves that if the dynamics on $X$ is generated by
more than one map then the system cannot be super-transitive and
hence cannot exhibit any stronger forms of mixing. Further as
super-transitivity of the system $(X,F)$ is necessary for any
induced admissible system $(\Psi,\overline{F})$ to be
transitive\cite{ap}, the system induced on the hyperspace by a
family of more than one map cannot be transitive and hence cannot
exhibit stronger notions of mixing. In light of the remark stated,
we obtain the following corollary.
\end{Remark}

\begin{Cor}
Let $(X,F)$ be a dynamical system generated by a finite commutative
family of continuous self maps on X. Let $(\Psi,\Delta)$ be the
associated hyperspace and let $\overline{F}$ be the corresponding
induced map. If the family $F$ contains more than one map then
$\overline{F}$ cannot be transitive.
\end{Cor}

We now show that the dynamics induced by a commutative family on the
hyperspace cannot exhibit dense set of periodic points.

\begin{Proposition}
Let $(X,F)$ be a dynamical system generated by a finite commutative
family of continuous self maps on X and let $(\mathcal{K}(X),
\overline{F})$ be the induced system endowed with the Vietoris
topology on the hyperspace. If $(\mathcal{K}(X), \overline{F})$
exhibits dense set of periodic points then $F$ is a singleton.
\end{Proposition}

\begin{proof}
Let if possible, the induced mao $\overline{F}$ exhibit dense set of
periodic points. Let $x\in X$ be arbitrary and let $\epsilon>0$ be
given. For any two members $f$ and $g$ of $F$, as $f$ and $g$ are
continuous on a compact set, they are uniformly continuous, i.e.
there exists $\delta>0$ such that whenever $d(x,y)<\delta$, we have
$d(f(x),f(y))<\frac{\epsilon}{3}$ and
$d(g(x),g(y))<\frac{\epsilon}{3}$.

Let $U= S_{\frac{\delta}{2}}(x)$. As $<U>$ is open in the hyperspace
and $\overline{F}$ has dense set of periodic points, there exists
$A\in <U>$ and $n\in \mathbb{N}$ such that $\overline{F}^n(A)=A$.
Let $a\in A$. As $\overline{F}^n(A)=A\subset U$ we have
$d(x,f^n(a))<\frac{\delta}{2}$. Consequently, we have $d(f(x),
f^{n+1}(a))<\frac{\epsilon}{3}$ and $d(g(x), g\circ
f^{n}(a))<\frac{\epsilon}{3}$. Also $\overline{F}^n(A)=A\subset U$,
implies $d(f^n(a), g\circ f^{n-1}(a))<\delta$ and hence
$d(f^{n+1}(a), g\circ f^n(a))<\frac{\epsilon}{3}$ (as $F$ is
commutative). Using triangle inequality we have, $d(f(x),g(x)) \leq
d(f(x),f^{n+1}(a))+ d(f^{n+1}(a), g\circ f^n(a))+ d(g\circ f^n(a),
g(x)) < \epsilon$.\\

As $\epsilon>0$ was arbitrary, $d(f(x),g(x))=0$ which implies $f(x)=
g(x)$. As the proof holds for any $x\in X$, any two members of the
family $F$ coincide and hence $F$ is a singleton.
\end{proof}

\begin{Remark}
The result establishes that if the dynamics on the space $X$ is
generated by more than one map then the induced dynamics on the
hyperspace cannot exhibit dense set of periodic points. The proof
uses the openness of the sets of the form $<U>$, where $U$ is
non-empty open in $X$ and does utilize the complete structure of the
Vietoris topology. Further, the proof does not utilize the structure
of the hyperspace $\mathcal{K}(X)$ and holds good for any general
admissible hyperspace $(\Psi,\Delta)$. Hence the induced map cannot
have dense set of periodic points in this case and the result is
true for the induced system when the admissible hyperspace is
endowed with topology finer than the upper Vietoris topology. In
light of the remark stated, we get the following corollary.
\end{Remark}

\begin{Cor}
Let $(X,F)$ be a dynamical system generated by a finite commutative
family of continuous self maps on X and let $(\Psi, \overline{F})$
be the induced system endowed with any hyperspace topology finer
than the upper Vietoris topology. If the family $F$ contains more
than one map then $\overline{F}$ cannot exhibit dense set of
periodic points.
\end{Cor}

\begin{Remark}
The above corollary establishes the generalization of the proved
result to the system induced on a more general admissible
hyperspace. Further, it is worth mentioning that if any of the
members of $F$ exhibit dense set of periodic points, then the system
$(X,F)$ exhibits dense set of periodic points as $f^n(x)\in F^n(x)$
for any member $f$ of $F$. Hence if the dynamics on $X$ is induced
by more than one function, the system may exhibit dense set of
periodic points but the induced system can not exhibit dense set of
periodic points which is contrary to the case when $F$ is a
singleton. In the light of the remark stated, we get the following
corollary.
\end{Remark}

\begin{Cor}
Let $(X,F)$ be a dynamical system and let $(\Psi, \overline{F})$ be
the corresponding induced system.  Then, $(X,F)$ has dense set of
periodic points $\nRightarrow$ $(\Psi, \overline{F})$ has dense set
of periodic points.
\end{Cor}

\begin{Remark}
The above corollary shows that there can exist dynamical systems
$(X,F)$ with the dense set of periodic points such that the induced
system does not have dense set of periodic points. Such a scenario
happens when the dynamics on $X$ is induced by more than one map.
Such a behavior for the induced system is due to the fact that
although the dynamics on the space $X$ is generated by a finite
commutative relation $F$, the dynamics on the hyperspace is
generated by a function and hence conventional methodology for
investigating the dynamics on the hyperspace are to be used. In
particular, $x$ is periodic for $F$ if there exists $n\in
\mathbb{N}$ such that $x\in F^n(x)$ but $A$ is periodic for
$\overline{F}$ if there exists $n\in \mathbb{N}$ such that $A=
\overline{F}^n(A)~~ (= F^n(A))$. Consequently, there is a
significant difference in the basic dynamical notions of the system
which induces such a contrasting  behavior on the hyperspace. We now
give an example in support of our argument.
\end{Remark}

\begin{ex}
Let $X= S^1$ be the unit circle and let $f_1$ and $f_2$ be the
rational rotations on unit circle defined as $f_1(\theta)= \theta+
p$ and $f_2(\theta)= \theta+q$ respectively.

Let $F=\{f_1,f_2\}$ be the commutative relation  and let $(X,F)$ be
the corresponding dynamical system generated. As $f_i$ is rotation
on the unit circle by a rational multiple of $2\pi$, $f_i$ exhibits
dense set of periodic points. Further as $f_i^n(x)\in F^n(x)$ for
all $i$ and $n$, $F$ exhibits dense set of periodic points. However,
as each $f_i$ is a rotation on the unit circle (by different
angles), for any set $A\neq X$ $F^n(A)=A$ never holds good and hence
the hyperspace has a unique periodic point $X$. Consequently, the
example shows that there exists dynamical systems $(X,F)$ with the
dense set of periodic points such that the induced system does not
have dense set of periodic points.
\end{ex}

\begin{Remark}
The above example proves that the induced dynamics cannot exhibit
dense set of periodic points when induced by a commutative family of
more than one map. On similar lines, considering $f_i$ ($i=1,2$) to
be irrational rotations on the unit circle gives an example of a
relation where each component is transitive but the relation is not
super-transitive. The relation is not super-transitive as $f_1$ and
$f_2$ are isometries and hence for each natural number $2n$ (or
$2n+1$) $f_1^n\circ f_2^n$ and $f_1\circ f_2^{2n-1}$  (or
$f_1^n\circ f_2^n$ and  $f_1\circ f_2^{2n}$) cannot be pushed inside
an open set of arbitrary arc length. Consequently, the dynamics
induced on the hyperspace by the family considered is not
super-transitive.
\end{Remark}

\begin{Remark}
The above results establish that when the dynamics on $X$ is
generated by a commutative family of more than one function, the
dynamics on the hyperspace cannot be super-transitive and cannot
have dense set of periodic points. The commutativity of the family
$F$ plays an important role in proving the result and hence cannot
be dropped. In absence of commutativity, the order of $f$ and $g$ in
any member of $F^n(x)$ cannot be altered and the proof of the result
does not hold good. However under non-commutativity, $g\circ f^n(x)$
and $f\circ g\circ f^{n-1}(x)$ (and others where $f_i$'s are applied
same number of times but in different order) do not coincide and
hence $F^n(x)$ contains more elements. Consequently, it is expected
that as super-transitivity (or dense periodicity) does not hold
under commutativity, it will not hold in the non-commutative case
(as $F^n(x)$ contains more elements). However, any proof for the
belief is not available and hence is left open.
\end{Remark}

We now give an example to show, that the induced dynamics, when
induced by more than one function, might be sensitive.

\begin{ex}
Let $\Sigma_2$ be the sequence space of all bi-finite sequences of
two symbols $0$ and $1$. For any two sequences $x= (x_i)$ and $y=
(y_i)$, define

\centerline{$d(x,y) = \sum \limits_{i= -\infty}^{\infty} \frac{|x_i
- y_i |}{2^{|i|}}$}

It is easily seen that the metric $d$ generates the product topology
on $\sum_2$. \\


\centerline{Let $ \sigma : \Sigma_2 \rightarrow \Sigma_2 $ be
defined as}
\centerline{$ \sigma(\ldots x_{-2} x_{-1}.x_0 x_1
\ldots) = \ldots x_{-2} x_{-1} x_{0}. x_1 x_2 x_3 \ldots $}

\vskip 0.3cm

The map $\sigma$ is known as the shift map and is continuous with
respect to the metric $d$ defined.

Let $\mathbb{F}=\{\sigma,\sigma^2\}$ and let $\overline{F}$ be the
corresponding induced map on the hyperspace. We claim that the
induced map is sensitive on $\mathcal{K}(X)$. Equivalently, it is
sufficient to show that the map is sensitive on a dense subset of
$\mathcal{K}(X)$.

Let $A=\{x_1,x_2,\ldots,x_k\}$ be a finite set where each $x_i=
(\ldots x_i^{-2} x_i^{-1} .x_i^0 x_i^1\ldots x_i^n \ldots)$ is an
element in $\Sigma$ and let $\epsilon>0$ be given. Let $r\in
\mathbb{N}$ such that $\frac{1}{2^r}<\epsilon$.

Let $y_i= (\ldots x_i^{-2} x_i^{-1}.x_i^0 \ldots x_i^{r+1}
000\ldots)$ and $z_i= (\ldots x_i^{-2} x_i^{-1} x_i^0 \ldots
x_i^{r+1} 111\ldots)$. Then $B=\{y_1,y_2,\ldots,y_k\}$ and
$C=\{z_1,z_2,\ldots,z_k\}$ are elements in $S(A,\epsilon)$ such that
$F^{r+2}(y_i)=(\ldots x_i^{-2}x_i^{-1}\ldots x_i^r.0000\ldots)$ and
$F^{r+2}(z_i)=(\ldots x_i^{-2}x_i^{-1}\ldots x_i^r.11111\ldots)$.

Consequently, $d_H(\overline{F}^{r+2}(B), \overline{F}^{r+2}(C))\geq
1$ and hence $\overline{F}$ is sensitive at $A$. As the proof holds
for any finite subset $S$ of $Sigma_2$, $\overline{F}$ is sensitive
at finite subsets of $\Sigma$. Consequently, $\overline{F}$ is
sensitive on $\Sigma$.
\end{ex}

\section{Conclusion}

The paper discusses the dynamics of the induced function on the
hyperspace, when the function is induced by a non-trivial family of
commuting continuous self maps on $X$. It is observed that the
dynamics is contrary to the case when the map on the hyperspace is
induced using a single function. While the map induced by single
function can exhibit complex dynamical behavior (for example weakly
mixing or topological mixing), the dynamics induced by a collection
of two or more commuting maps cannot even be transitive and hence
cannot exhibit any of the higher notions of mixing. Further, it is
established that the dynamics induced by such a family cannot have
dense set of periodic points. This once again is a contrary to the
case when the map is induced by a single function, as the map
induced in that case always has dense set of periodic points, if the
original system has dense set of periodic points. We also give an
example to show that the dynamics induced by a commutative family
may be sensitive. It is worth mentioning that although the problem
has been solved when the dynamics on the hyperspace is induced by a
commutative family, the non-commutative case is still open for
investigation. As $F^n(x)$ contains more points in the
non-commutative case,  similar results are expected to hold. However
as the proof for non-commutative case could not be derived, we leave
it open for further investigation.

%
%
%
%

%

\bibliography{xbib}

\end{document}